\documentclass[12pt,reqno]{amsart}

\usepackage{amsbsy,amsfonts,amsmath,amssymb,amscd,amsthm,mathrsfs,verbatim,calc}
\usepackage{graphicx,epsfig,color}
\usepackage{pgf,tikz,pgfplots}
\usepackage{enumitem}
\usepackage[a4paper,text={6.1in,8in},centering,includefoot,foot=1in]{geometry}
\usepackage[colorlinks=true,hypertexnames=false,linkcolor=green,citecolor=green]{hyperref}

\small\normalsize
\theoremstyle{plain}

\newtheorem{prelem}{{\bf Proposition}}

\newtheorem{theorem}{Theorem}
\newtheorem{corollary}[theorem]{Corollary}
\newtheorem{lemma}[theorem]{Lemma}
\newtheorem{remarks}[theorem]{Remarks}

\newtheorem{proposition}[theorem]{Proposition}

\numberwithin{equation}{section}
\let\oldmarginpar\marginpar
\renewcommand\marginpar[1]{\-\oldmarginpar[\raggedleft\footnotesize \textcolor{red}{#1}]{\raggedright\footnotesize\textcolor{red}{#1}}}

\begin{document}

\title[Total Roman bondage number of a graph]{Total Roman bondage number of a graph}
\author[]{  F. Khosh-Ahang Ghasr     and S. Nazari-Moghaddam\\\vspace{2mm}\\Department of Mathematics, Faculty of Science, \\ Ilam University,
Ilam, Iran.\\\texttt{ f.khoshahang@ilam.ac.ir, sakine.nazari.m@gmail.com}\vspace{4mm}
  }

\begin{abstract}
A \emph{total Roman dominating function} (TRDF) on a graph $G$ with no isolated vertices is a function $f:V(G)\to\{0,1,2\}$ such that every vertex $v$ with $f(v)=0$ has a neighbor assigned $2$, and the subgraph induced by $\{v:f(v)>0\}$ has no isolated vertices. The \emph{total Roman domination number} $\gamma_{tR}(G)$ is the minimum weight of a TRDF on $G$. The \emph{total Roman bondage number} $b_{tR}(G)$ is the minimum cardinality of an edge set $E'\subseteq E(G)$ such that $G-E'$ has no isolated vertices and $\gamma_{tR}(G-E')>\gamma_{tR}(G)$; if no such $E'$ exists, $b_{tR}(G)=\infty$.

We prove that deciding whether $b_{tR}(G)\leq k$ is NP-complete for arbitrary graphs. We establish sharp bounds, including $\gamma_{tR}(G)+1\leq \gamma_{tR}(G-B)\leq \gamma_{tR}(G)+2$ for any $b_{tR}(G)$-set $B$ (both sharp), and $b_{tR}(G)\geq \max\{\delta(G),b(G)\}$ when $\gamma_{tR}(G)=3\beta(G)$. We characterize graphs with $b_{tR}(G)=\infty$ and provide a necessary and sufficient condition for $b_{tR}(G)=1$. Exact values are determined for complete graphs, complete bipartite graphs, brooms, double brooms, wheels and wounded spiders. 

Further upper bounds are given in terms of order, diameter, girth, and structural features.
\end{abstract}


\maketitle

\setcounter{tocdepth}{1}


\section{Introduction}
Throughout this paper, $G=(V,E)$ is a simple graph with no isolated vertices. Its order $|V|$ and size $|E|$ are denoted by $n$ and $m$ respectively. For a vertex $v\in V$, the \textit{open neighborhood} $N(v)$ is $\{u\in V:uv\in E\}$, the \textit{closed neighborhood} $N[v]=N(v)\cup\{v\}$, and the \textit{degree} $\deg(v)=|N(v)|$. These extend naturally to subsets $S\subseteq V$. A vertex of degree $1$ is a \textit{leaf}.  A vertex which is adjacent to a leaf is called a \textit{support vertex}.  It is called a \textit{strong support vertex} if $|L(v)|\geq 2$ where $L(v)$ is the set of leaves adjacent to $v$. The minimum and maximum degrees are $\delta=\delta(G)$ and $\Delta=\Delta(G)$. For $S\subseteq V$, the \textit{$S$-external private neighborhood} of $v\in S$ is $\mathrm{epn}(v,S)=\{w\in V\setminus S:N(w)\cap S=\{v\}\}$. The \textit{distance} $d(u,v)$ is the length of a shortest $u$-$v$ path in a connected graph, and the \textit{diameter} $\mathrm{diam}(G)$ is the maximum distance. The \textit{girth} $\mathrm{girth}(G)$ is the length of a shortest cycle. 

A \textit{(total) dominating set} $S$ satisfies $N[S]=V$ (resp.\ $N(S)=V$), with \textit{(total) domination number} $\gamma(G)$ (resp.\ $\gamma_t(G)$) its minimum size. The \textit{bondage number} $b(G)$ is the minimum $|E'|$ such that $\gamma(G-E')>\gamma(G)$, introduced in \cite{refrefref}. Kulli and Patwari \cite{kul} defined the \textit{total bondage number} $b_t(G)$ similarly for $\gamma_t$.

A \textit{Roman dominating function} (RDF) $f:V\to\{0,1,2\}$ requires every $v$ with $f(v)=0$ adjacent to some $u$ with $f(u)=2$ \cite{rr,s}. The \textit{Roman domination number} $\gamma_R(G)$ is the minimum weight $\omega(f)=|V_1|+2|V_2|$, where $(V_0,V_1,V_2)$ partitions $V$ by $f$-values. Roman domination is widely studied (e.g., \cite{R1,Co,R2}). The \textit{Roman bondage number} $b_R(G)$ is the minimum $|E'|$ with $\gamma_R(G-E')>\gamma_R(G)$ \cite{rad}.

A \textit{total Roman dominating function} (TRDF) is an RDF where $G[\{v:f(v)>0\}]$ has no isolated vertices. A \textit{quasi-total Roman dominating function} (QTRDF) further requires no isolated $v$ with $f(v)=2$. The \textit{total} (resp.\ \textit{quasi-total}) \textit{Roman domination number} $\gamma_{tR}(G)$ (resp.\ $\gamma_{qtR}(G)$) is the minimum TRDF (resp.\ QTRDF) weight \cite{GG,LC}, extensively investigated (e.g., \cite{ah1,MV,Mar}). The \textit{quasi-total Roman bondage number} $b_{qtR}(G)$ is defined analogously \cite{J}.

 In this paper we study the \textit{total Roman bondage number} $b_{tR}(G)$ which is the minimum cardinality of a subset $E'$ of $E$ such that $G-E'$ has no isolated vertices and \linebreak $\gamma_{tR}(G-E')>\gamma_{tR}(G)$. If no such $E'$ exists, then set $b_{tR}(G)=\infty$. 

After preliminaries in  Section~2, in Section~3 we  prove that the decision problem for $b_{tR}(G)$ is NP-hard via reduction from 3-SAT. In Section 4 we present properties, bounds and exact values for total Roman bondage number for a few classes of graphs.

\section{Preliminaries}
We start by recalling some special classes of graphs needed in the remainder of the paper. We denote the \textit{complete graph}, \textit{path}, and \textit{cycle} of order $n$ by $K_n$, $P_n$, and $C_n$, respectively. If $C$ is a cycle with vertices $v_1,\dots,v_n$ and edges $v_1v_n$, $v_iv_{i+1}$ for $i=1,\dots,n-1$, we sometimes write $C=(v_1,\dots,v_n)$. The \textit{wheel graph} $W_{n+1}$ is obtained by joining a new vertex to all vertices of $C_n$. We write $K_{p,q}$ for the \textit{complete bipartite graph} with bipartition $(X,Y)$ where $|X|=p$ and $|Y|=q$. A complete subgraph of $G$ is a \textit{clique}. The \textit{corona} of a graph $H$, denoted $\mathrm{cor}(H)$, is obtained by adding a pendant edge to each vertex of $H$.

The \textit{centers} of a graph are the vertices of minimum eccentricity. A tree with exactly one center and $t\geq 2$ leaves is a \textit{star} $S_t$. A tree of diameter three with exactly two centers, one adjacent to $r$ leaves and the other to $s$ leaves, is a \textit{bistar} $S_{r,s}$. For integers $t\geq 2$ and $k\geq 1$, a \textit{wounded spider} $S(k,t)$ is a star $S_t$ with $t-k$ edges subdivided once. A \textit{healthy spider} $S(0,t)$ is a star with all edges subdivided (also called a \textit{subdivided star}). In the wounded or healthy spider, the degree-$t$ vertex is the \textit{head}, vertices at distance two from the head are \textit{feet}, and those at distance one (if any) are \textit{wounded feet}. A \textit{broom} $B(t,d)$ ($t\geq 3$, $d\geq 2$) has $t+d$ vertices: an induced path $x_1,\dots,x_t$ with $d$ pendants adjacent to $x_t$ (or sometimes $x_1$; we specify if needed). A \textit{double broom} $B(t,d,d')$ ($t\geq 3$, $d,d'\geq 2$) has path $x_1,\dots,x_t$ with $d$ pendants at $x_1$ and $d'$ at $x_t$.

The \textit{vertex cover number} $\beta(G)$ is the minimum number of vertices incident to all edges.

\begin{prelem}\label{AA}
(\cite[Theorems~4, 5, 6 and Proposition~7]{ah1}, \cite[Relation~(1)]{GG}, \cite[Proposition~1]{Co}) Let $G$ be a graph with no isolated vertex and let $H$ be any graph. Then
\begin{enumerate}
\item $\gamma(G)\leq  \beta (G)$.
\item $\gamma (G)\leq \gamma _t(G)$.
\item $\gamma _R(G)\leq \gamma_{qtR}(G) \leq \gamma_{tR}(G)$.
\item $\gamma _t (G)\leq \gamma _{tR}(G)\leq 2\gamma _t (G)$.
\item $2\gamma (G)\leq \gamma _{tR}(G)\leq 3\gamma (G)$.
\item $\gamma (H)\leq \gamma _{R}(H)\leq 2\gamma  (H)$.
\item For connected graph $G$ of order $n\geq 3$, $\gamma _{tR}(G)=\gamma _t(G)+1$ if and only if $\Delta (G)=n-1$.
\item $\gamma_{tR}(G)=\gamma_t(G)$ if and only if $G$ is the disjoint union of copies of $K_2$.
\end{enumerate}
\end{prelem}

The next lemma is immediate from Proposition~\ref{AA}.
\begin{lemma}\label{m2}
For any graph $G$ with no isolated vertices,
\begin{enumerate}
\item if $\gamma (G)=\gamma_t(G)$, then $b_t(G)\leq b(G)$;
\item if $\gamma _R (G)=\gamma_{qtR}(G)$, then $b_{qtR}(G)\leq b_R(G)$;
\item if $\gamma _{qtR}(G)=\gamma_{tR}(G)$, then $b_{tR}(G)\leq b_{qtR}(G)$;
\item if $\gamma _R(G)=\gamma_{tR}(G)$, then $b_{tR}(G)\leq b_{qtR}(G)\leq b_R(G)$;
\item if $\gamma _t(G)=\gamma_{tR}(G)$, then $b_{tR}(G)\leq b_t(G)$;
\item if $\gamma_{tR}(G)=2\gamma _t(G)$, then $b_t(G)\leq b_{tR}(G)$;
\item if $2\gamma (G)=\gamma_{tR}(G)$, then $b_{tR}(G)\leq b(G)$;
\item if $\gamma _{tR}(G)=3\gamma (G)$, then $b(G)\leq b_{tR}(G)$;
\item if $\gamma (G)=\gamma _R(G)$, then $b_R(G)\leq b(G)$;
\item (\cite[Theorem~5.4]{HuXu}) If $\gamma _R(G)=2\gamma (G)$, then $b(G)\leq b_R(G)$.
\end{enumerate}
\end{lemma}

Let $\mathcal{G}$ be the family of unicyclic graphs obtained from $C_4 = (v_1, v_2, v_3, v_4)$
by attaching $k_1$ copies of $P_3$ at $v_1$ and $k_2$ copies of $P_3$ at $v_2$ where $k_1 \geq 0$ and $k_2 \geq 0$. Let $\mathcal{H}$ be the family of trees obtained from a bistar by
subdividing each pendant edge once and the central edge $r$ times where $r \geq 0$.

\begin{prelem}\label{ahah1}
(\cite[Theorem~13]{ah1}) Let $G$ be a connected graph of order $n$. Then $\gamma_{tR}(G)=n$ if and only if one of the following holds.
\begin{enumerate}
\item $G$ is a path or cycle;
\item $G=\mathrm{cor}(F)$ for some graph $F$;
\item $G$ is a subdivided star;
\item $G\in\mathcal{G}\cup\mathcal{H}$.
\end{enumerate}
\end{prelem}

\begin{prelem}\label{th4} 
(\cite[Theorem~6]{ah1}) For $G$ with no isolated vertices, $\gamma_t(G)=\gamma_{tR}(G)$ if and only if $G$ is a disjoint union of $K_2$'s.
\end{prelem}

\begin{corollary}\label{infty}
For each graph $G$ with no isolated vertices, $b_{tR}(G)=\infty$ if and only if every component is
\begin{enumerate}
\item a star;
\item a healthy spider or wounded spider with one wounded foot;
\item a path or cycle;
\item a corona of a connected graph;
\item any graph in $\mathcal{G}\cup\mathcal{H}$.
\end{enumerate}
In particular, if $\gamma_{tR}(G)=\gamma_t(G)$, then $b_{tR}(G)=\infty$.
\end{corollary}
\begin{proof}
If $G_1,\dots,G_k$ are the components of $G$, then $\gamma_{tR}(G)=\sum\gamma_{tR}(G_i)$.
So in view of Proposition \ref{ahah1}, it is enough to prove the \textit{only if } part. Suppose $b_{tR}(G)=\infty$. Then for every $E'\subseteq E(G)$, either $G-E'$ has an isolated vertex or $\gamma_{tR}(G-E')=\gamma_{tR}(G)$. If $G$ is connected and every single-edge removal creates an isolated vertex, then $G$ is a star. Otherwise, there exists $E'$ with no isolated vertices in $G-E'$ but every single-edge removal from $G-E'$ creates an isolated vertex, so $G-E'$ is a disjoint union of stars $G_1,\dots,G_k$. Therefore 
$$\gamma_{tR}(G)=\gamma_{tR}(G-E')=\sum_{1\leq i \leq k}\gamma_{tR}(G_i)=\sum_{1\leq i \leq k}|V(G_i)|=n.$$
So each component has $\gamma_{tR}=|V|$, and Proposition~\ref{ahah1} applies. The final claim follows from Proposition~\ref{th4}.
\end{proof}

\begin{prelem}\label{Pro} 
 (\cite[Proposition~2.3]{chloe}) If $G$ has no isolated vertices and $uv\in E(\overline{G})$, then $\gamma_{tR}(G)-2\leq \gamma_{tR}(G+uv)\leq \gamma_{tR}(G)$.
\end{prelem}

 The following Remarks are straightforward.
\begin{remarks}\label{Remarks}
For any graph $G$ with no isolated vertices,
\begin{enumerate}
\item every support vertex $v$ satisfies $f(v)\geq 1$ in any TRDF $f$;
\item for $n\geq 3$, $\gamma_{tR}(G)=3$ if and only if $\Delta(G)=n-1$;
\item for $n\geq 3$ with exactly $t\geq 1$ vertices of degree $n-1$, $b_{tR}(G)=\lceil t/2\rceil$ (by adapting \cite[Lemma~3.1]{HuXu} and (2));
\item $b_{tR}(K_n)=\lceil n/2\rceil$ and $b_{tR}(W_{n+1})=1$;
\item $\gamma_t(G)=2$ if and only if there are adjacent vertices $u,v$ with $N[u]\cup N[v]=V(G)$;
\item for any (double) broom, $b_{tR}(G)=1$.
\end{enumerate}
\end{remarks}

\begin{proposition}\label{4R}
For a graph $G$ of order $n\geq 3$ which has no isolated vertices, $\gamma _{tR}(G)=4$ if and only if $G=2K_{2}$ or $\Delta (G)\leq n-2$ and there exist adjacent vertices $u$ and $v$ with $N[u]\cup N[v]=V(G)$.
\end{proposition}
\begin{proof}
Let $\gamma_{tR}(G)=4$ and $f=(V_0,V_1,V_2)$ be a $\gamma_{tR}$-function on $G$. By Remarks \ref{Remarks}(2), $\Delta(G)\leq n-2$. 

\textit{Case 1:} $|V_2|=0$. Then $V_0=\emptyset$, $|V_1|=4$, $n=4$. As $G$ has no isolated vertices and $\Delta(G) \leq 2$, possible graphs are $2K_2$, $P_4$, $C_4$.

\textit{Case 2:} $|V_2|=1$. Let $V_2=\{v\}$, $|V_1|=2$. Then $V_0$ is dominated only by $v$. Also totality and $\Delta(G)\leq n-2$,  requires $v$ adjacent to exactly one in $V_1$, with the two in $V_1$ adjacent.

\textit{Case 3:} $|V_2|=2$. Then $V_1=\emptyset$, the two in $V_2$ are adjacent (totality), and dominate $V_0$.

In each case, adjacent $u,v$ cover all vertices (or $G=2K_2$).

Conversely, if $G=2K_2$, clearly $\gamma_{tR}=4$. Otherwise, $\Delta\leq n-2$ and there exist adjacent vertices $u,v$ with $N[u]\cup N[v]=V$ imply $\gamma_{tR}\geq 4$ (by Remark~\ref{Remarks}(2)). The function $f=(V(G)\setminus\{u,v\},\emptyset,\{u,v\})$ is a TRDF of weight 4 and so $\gamma _{tR}(G)=4$.
\end{proof}

Let $\mathcal{B}$ be the class of graphs $G$ with no isolated vertices such that $b_{tR}(G)<\infty$; equivalently, no component is listed in Corollary~\ref{infty}. Henceforth, all graphs belong to $\mathcal{B}$.

\section{NP-hardness of total Roman bondage}

Our aim in this section is to show that the decision problem associated with
the total Roman bondage number is NP-hard for general graphs. \newline
\newline
\textbf{Total Roman Bondage problem (TR-Bondage)}\newline
\newline
\textbf{Instance: }A nonempty graph $G$ and a positive integer $k$.\newline
\textbf{Question: }Is $b_{tR}(G)\leq k?$

\bigskip

We show the NP-hardness of TR-Bondage problem by transforming the
3-satisfiability problem (3-SAT problem) to it in polynomial time. Recall
that the 3-SAT problem specified below was proven to be NP-complete in \cite%
{GJ}. \newline
\newline
\textbf{3-SAT problem}\newline
\newline
\textbf{Instance:} A collection $\mathcal{C}=\{C_1, C_2, \dots , C_m\}$ of
clauses over a finite set $U$ of variables such that $\left\vert C_j\right\vert =3$ for $j=1,2,\dots ,m.$\newline
\textbf{Question:} Is there a truth assignment for $U$ that satisfies all
the clauses in $\mathcal{C}$?

\begin{theorem}
The TR-Bondage problem is NP-hard.
\end{theorem}

\begin{proof}
Let $U=\{u_1, u_2, \dots , u_n\}$ and $\mathcal{C}%
=\{C_1, C_2, \dots , C_m\}$ be an arbitrary instance of the 3-SAT problem.
We will construct a graph $G$ and a positive integer $k$ such that $\mathcal{%
C}$ is satisfiable if and only if $b_{tR}(G)\leq k.$

\begin{figure}[ht]
\centering
\begin{tikzpicture}[scale=.8, transform shape]
\node [draw, shape=circle,fill=black,scale=0.5] (a1) at  (0,0) {};
\node [draw, shape=circle,fill=black,scale=0.5] (a2) at  (1,0) {};
\node [draw, shape=circle,fill=black,scale=0.5] (a3) at  (2,0) {};
\node [draw, shape=circle,fill=black,scale=0.5] (a4) at  (0,1) {};
\node [draw, shape=circle,fill=black,scale=0.5] (a5) at  (1,1) {};
\node [draw, shape=circle,fill=black,scale=0.5] (a6) at  (2,1) {};
\node [draw, shape=circle,fill=black,scale=0.5] (a7) at  (1,2) {};
\draw(a1)--(a3)--(a6)--(a1)--(a4)--(a3)--(a6)--(a7)--(a4)--(a6);
\node at (0,-.4) {$u_i$};
\node at (1,-.4) {$t_i$};
\node at (2,-.4) {$\overline{u_i}$};
\node at (0,1.4) {$a_i$};
\node at (1,1.4) {$s_i$};
\node at (2,1.4) {$b_i$};
\node at (1,2.4) {$d_i$};
\node [draw, shape=circle,fill=black,scale=0.5] (a1) at  (3,1) {};
\node [draw, shape=circle,fill=black,scale=0.5] (a2) at  (4,0) {};
\node [draw, shape=circle,fill=black,scale=0.5] (a3) at  (5,1) {};
\node [draw, shape=circle,fill=black,scale=0.5] (a4) at  (4,2) {};
\draw(a1)--(a3)--(a2)--(a1)--(a4)--(a3);
\node at (2.7,1) {$q$};
\node at (5.3,1) {$p$};
\node at (4,-.3) {$o$};
\node at (4,2.3) {$r$};
\end{tikzpicture}
\caption{The graphs $H_i$ and $F$.}
\label{fig-1}
\end{figure}
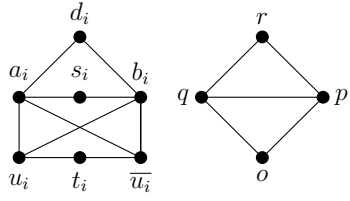

For each variable $u_i\in U,$ associate the connected graph $H_{i}$ as
shown in Figure 1. Corresponding to each clause $C_{j}=\{x_j, y_j, z_j\}\in \mathcal{C}$, associate a single vertex $c_{j}$ and add the edge-set $E_j=\{c_jx_j,c_jy_j,c_jz_j\}$. Next add the graph $F$ and join
vertex $r$ to each $c_j$, and let $G$ be the resulting graph. Clearly $G$
is a graph of order $7n+m+4$ and size $10n+4m+5$.  An example of the graph $G$ when $U=\{u_1, u_2, u_3, u_4\}$ and $\mathcal{C}=\{\{u_1, u_2, \overline{u_3}\}, \{u_2, \overline{u_3}, \overline{u_4}\}\}$ is illustrated in Figure \ref{fig-2}. 

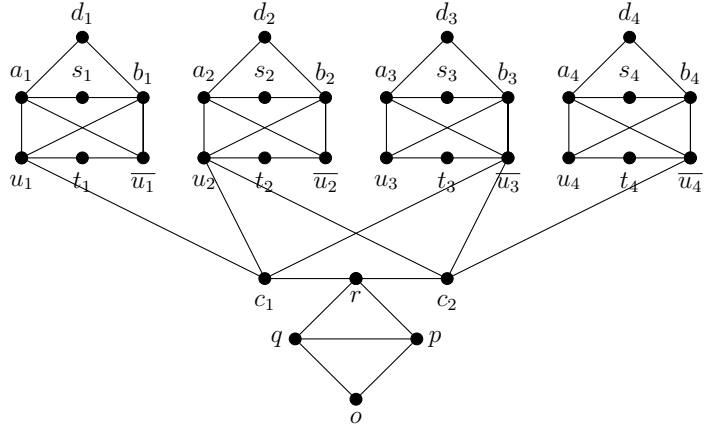
\begin{figure}[ht]
\centering
\begin{tikzpicture}[scale=.8, transform shape]
\node [draw, shape=circle,fill=black,scale=0.5] (a1) at  (0,0) {};
\node [draw, shape=circle,fill=black,scale=0.5] (a2) at  (1,0) {};
\node [draw, shape=circle,fill=black,scale=0.5] (a3) at  (2,0) {};
\node [draw, shape=circle,fill=black,scale=0.5] (a4) at  (0,1) {};
\node [draw, shape=circle,fill=black,scale=0.5] (a5) at  (1,1) {};
\node [draw, shape=circle,fill=black,scale=0.5] (a6) at  (2,1) {};
\node [draw, shape=circle,fill=black,scale=0.5] (a7) at  (1,2) {};
\draw(a1)--(a3)--(a6)--(a1)--(a4)--(a3)--(a6)--(a7)--(a4)--(a6);
\node at (0,-.4) {$u_1$};
\node at (1,-.4) {$t_1$};
\node at (2,-.4) {$\overline{u_1}$};
\node at (0,1.4) {$a_1$};
\node at (1,1.4) {$s_1$};
\node at (2,1.4) {$b_1$};
\node at (1,2.4) {$d_1$};

\node [draw, shape=circle,fill=black,scale=0.5] (a1) at  (3,0) {};
\node [draw, shape=circle,fill=black,scale=0.5] (a2) at  (4,0) {};
\node [draw, shape=circle,fill=black,scale=0.5] (a3) at  (5,0) {};
\node [draw, shape=circle,fill=black,scale=0.5] (a4) at  (3,1) {};
\node [draw, shape=circle,fill=black,scale=0.5] (a5) at  (4,1) {};
\node [draw, shape=circle,fill=black,scale=0.5] (a6) at  (5,1) {};
\node [draw, shape=circle,fill=black,scale=0.5] (a7) at  (4,2) {};
\draw(a1)--(a3)--(a6)--(a1)--(a4)--(a3)--(a6)--(a7)--(a4)--(a6);
\node at (3,-.4) {$u_2$};
\node at (4,-.4) {$t_2$};
\node at (5,-.4) {$\overline{u_2}$};
\node at (3,1.4) {$a_2$};
\node at (4,1.4) {$s_2$};
\node at (5,1.4) {$b_2$};
\node at (4,2.4) {$d_2$};

\node [draw, shape=circle,fill=black,scale=0.5] (a1) at  (6,0) {};
\node [draw, shape=circle,fill=black,scale=0.5] (a2) at  (7,0) {};
\node [draw, shape=circle,fill=black,scale=0.5] (a3) at  (8,0) {};
\node [draw, shape=circle,fill=black,scale=0.5] (a4) at  (6,1) {};
\node [draw, shape=circle,fill=black,scale=0.5] (a5) at  (7,1) {};
\node [draw, shape=circle,fill=black,scale=0.5] (a6) at  (8,1) {};
\node [draw, shape=circle,fill=black,scale=0.5] (a7) at  (7,2) {};
\draw(a1)--(a3)--(a6)--(a1)--(a4)--(a3)--(a6)--(a7)--(a4)--(a6);
\node at (6,-.4) {$u_3$};
\node at (7,-.4) {$t_3$};
\node at (8,-.4) {$\overline{u_3}$};
\node at (6,1.4) {$a_3$};
\node at (7,1.4) {$s_3$};
\node at (8,1.4) {$b_3$};
\node at (7,2.4) {$d_3$};

\node [draw, shape=circle,fill=black,scale=0.5] (a1) at  (9,0) {};
\node [draw, shape=circle,fill=black,scale=0.5] (a2) at  (10,0) {};
\node [draw, shape=circle,fill=black,scale=0.5] (a3) at  (11,0) {};
\node [draw, shape=circle,fill=black,scale=0.5] (a4) at  (9,1) {};
\node [draw, shape=circle,fill=black,scale=0.5] (a5) at  (10,1) {};
\node [draw, shape=circle,fill=black,scale=0.5] (a6) at  (11,1) {};
\node [draw, shape=circle,fill=black,scale=0.5] (a7) at  (10,2) {};
\draw(a1)--(a3)--(a6)--(a1)--(a4)--(a3)--(a6)--(a7)--(a4)--(a6);
\node at (9,-.4) {$u_4$};
\node at (10,-.4) {$t_4$};
\node at (11,-.4) {$\overline{u_4}$};
\node at (9,1.4) {$a_4$};
\node at (10,1.4) {$s_4$};
\node at (11,1.4) {$b_4$};
\node at (10,2.4) {$d_4$};

\node [draw, shape=circle,fill=black,scale=0.5] (a1) at  (4,-2) {};
\node [draw, shape=circle,fill=black,scale=0.5] (a2) at  (7,-2) {};
\node [draw, shape=circle,fill=black,scale=0.5] (a3) at  (0,0) {};
\node [draw, shape=circle,fill=black,scale=0.5] (a4) at  (3,0) {};
\node [draw, shape=circle,fill=black,scale=0.5] (a5) at  (8,0) {};
\node [draw, shape=circle,fill=black,scale=0.5] (a6) at  (11,0) {};
\node [draw, shape=circle,fill=black,scale=0.5] (a7) at  (5.5,-2) {};
\draw(a3)--(a1)--(a4)--(a2)--(a5)--(a1);
\draw(a6)--(a2)--(a7)--(a1);
\node at (4,-2.4) {$c_1$};
\node at (7,-2.4) {$c_2$};

\node [draw, shape=circle,fill=black,scale=0.5] (a1) at  (4.5,-3) {};
\node [draw, shape=circle,fill=black,scale=0.5] (a2) at  (5.5,-4) {};
\node [draw, shape=circle,fill=black,scale=0.5] (a3) at  (6.5,-3) {};
\node [draw, shape=circle,fill=black,scale=0.5] (a4) at  (5.5,-2) {};
\draw(a1)--(a3)--(a2)--(a1)--(a4)--(a3);
\node at (4.2,-3) {$q$};
\node at (6.8,-3) {$p$};
\node at (5.5,-4.3) {$o$};
\node at (5.5,-2.3) {$r$};
\end{tikzpicture}
\caption{An example of the graph $G$.}
\label{fig-2}
\end{figure}

Set $k=1$. Note that for every $\gamma _{tR}(G)$-function $f=(V_0, V_1, V_2)$, we
have $f(V(H_i))\geq 4$ for each $i\in \{1,2,\dots ,n\}$ and $f(V(F))\geq 3$. Therefore $\gamma _{tR}(G)\geq 4n+3$.\newline
\newline
\textbf{Claim 1.} $\gamma _{tR}(G)=4n+3$ if and only if $\mathcal{C}$ is satisfiable.\newline
\textit{Proof of Claim 1.} Assume that $\gamma _{tR}(G)=4n+3$ and let $f=(V_0, V_1, V_2)$ be a $\gamma _{tR}(G)$-function. Then $f(V(H_i))=4$ for each $i\in \{1,2,\dots ,n\}, f(V(F))=3$  and the value of $f$ on other vertices is zero. More precisely, $f(V(F))=3$ implies that $f(r)=1,f(o)=0$ and either ($f(p)=2$ and $f(q)=0$) or ($f(p)=0$ and $f(q)=2$). Moreover, $f(V(H_i))=4$ implies that $\left\vert V(H_i)\cap V_2\right\vert =2$ and $\left\vert V(H_i)\cap V_1\right\vert =0$ and thus we must have either $u_i,b_i\in V_2$ or
$\overline{u_i},a_i\in V_2$. Consequently, for each $i \in \{1, \dots , n\}$ we have $\left\vert \{u_i,\overline{u_i}\}\cap V_2\right\vert =1$. Now for every $j\in\{1, 2, \dots , m\}$, since $f(c_j)=0$ and $r\notin V_2$, we deduce that  $N(c_j)\cap \{u_i,\overline{u_i}\}\neq \emptyset $ for some $i\in \{1,2,\dots ,n\}$. Let $t:U\longrightarrow \{T,F\}$ be a mapping defined by $t(u_i)=T$ if $f(u_i)=2$ and $t(u_i)=F$ if $f(\overline{u_i})=2$.
Assume that $f(u_i)=2$ and $c_j\in N(u_i)$. By the construction of $G$, the literal $u_i$ is in the clause $C_j$. Then $t(u_i)=T$, implies that the clause $C_j$ is satisfied by $t$. Next assume that $f(\overline{u_i})=2$ and $c_j\in N(\overline{u_i})$. By the construction of $G$, the literal $\overline{u_i}$ is in the clause $C_j$. Then $t(u_i)=F$ and thus $t$ assigns $\overline{u_i}$ the true value $T$. Hence $t$ satisfies the clause $C_j$ and therefore $\mathcal{C}$ is satisfiable.

Conversely, assume that $\mathcal{C}$ is satisfiable, and let $t:U\longrightarrow \{T,F\}$ be a satisfying truth assignment for $\mathcal{C}$. We construct a TRDF $h$ on $G$ as follows. For every $i \in \{1, 2, \dots , n\}$ if $t(u_i)=T$, then let $h(u_i)=h(b_i)=2$, and if $t(u_i)=F$, then let $h(\overline{u_i})=h(a_i)=2$. Also, let $h(p)=2, h(r)=1$ and $h(x)=0$
for all remaining vertices. It is easy to check that $h$ is a TRDF on $G$ of weight $4n+3$. Thus $\gamma _{tR}(G)=4n+3$ and the proof of Claim 1 is complete.\newline
\newline
\textbf{Claim 2.} For every edge $e$ of $G$, $\gamma _{tR}(G-e)\leq 4n+4$.
\newline
\textit{Proof of Claim 2.} Let $e$ be an edge of $G$. If $e\in E(H_{i})$ for some $i\in \{1,2,\dots ,n\}$, define the function $g:V(G)\longrightarrow \{0,1,2\}$ by $g(r)=g(p)=2, g(u_\ell)=g(b_\ell)=2$ for every $\ell \in \{1,2,\dots ,n\}-\{i\}$.
Now if $e\in \{a_id_i, a_is_i, a_i\overline{u_i}, t_i\overline{u_i}\}$, then let $g(u_i)=g(b_i)=2$. If $e\in \{b_id_i, b_is_i, b_iu_i, t_iu_i\}$, then let $g(a_i)=g(\overline{u_i})=2$. If $e=a_iu_i$, then $g(\overline{u_i})=g(b_i)=2$ and eventually if $e=b_i\overline{u_i}$, then let $g(u_i)=g(a_i)=2$. (Of course notice that the cases $e\in \{b_id_i, a_id_i, a_iu_i, b_i\overline{u_i}\}$ can be included in other cases.)
Any remaining vertex of $G$ is assigned a $0$ under $g$. Clearly, $g$ is a TRDF on $G-e$ of weight $4n+4$. For the next, we assume that $e\notin E(H_{i})$ for every $i\in
\{1,2,\dots ,n\}$. If $e\in E(F)$, define the function $g:V(G)\longrightarrow \{0,1,2\}$ by $g(r)=g(u_{i})=g(b_{i})=2$ for every
$i\in \{1,2,\dots ,n\}$ and if $e\in \{rp, op\}$, then let $g(q)=2$ and else let $g(p)=2$. Also $g(x)=0$ for other vertices $x$ in $G$. Clearly, $g$ is a TRDF on $G-e$ of weight $4n+4$. Assume that
$e=rc_j$ for some $j\in \{1,2,\dots ,m\}$. Then since there exists $i\in \{1,2,\dots ,n\}$
such that $N(c_j)\cap \{u_i,\overline{u_i}\}\neq \emptyset$, without
loss of generality, let $c_ju_i\in E(G)$. Then assigning a $2$ to $r,p$
and every $u_i$ and $b_i$, and a $0$ to all remaining vertices provides
a TRDF on $G-e$ of weight $4n+4$. Finally, if $e$ is an edge linking
some $c_j$ with a variable of $U$, then the previous function is
TRDF on $G-e$ of weight $4n+4$. In any case, $\gamma _{tR}(G-e)\leq
4n+4$ and the proof of Claim 2 is complete.\newline
\newline
\textbf{Claim 3.} $\gamma _{tR}(G)=4n+3$ if and only if $b_{tR}(G)=1.$
\newline
\textit{Proof of Claim 3.} Assume that $\gamma _{tR}(G)=4n+3$. Let $f=(V_0, V_1, V_2)$ be a $\gamma _{tR}(G-e)$-function, where $e=pq$. Let $F'$ be the cycle $C_4$ induced by $o,p,q$ and $r$ in $G-e$.
Clearly, $f(V(H_i))\geq 4$ and so the total weight assigned for vertices
of all $H_i$'s equals to at least $4n$. Moreover, to total Roman dominate
vertices of $F'$, we need that $f(V(F'))+\overset{m}{\underset{j=1}{\sum }}f(c_j)\geq 4$ and thus $\gamma _{tR}(G-e)\geq 4n+4$. Therefore $b_{tR}(G)=1$.

Conversely, assume that $b_{tR}(G)=1$ and let $e$ be an edge of $G$ such
that $\gamma _{tR}(G-e)>\gamma _{tR}(G)$. By Claim 2, $\gamma _{tR}(G-e)\leq
4n+4$ and since $\gamma _{tR}(G)\geq 4n+3$ we deduce that $\gamma_{tR}(G)=4n+3$ which completes the proof of Claim 3.

According to Claims 1 and 3, we obtain that $b_{tR}(G)=1$ if and only if
there is a truth assignment for $U$ that satisfies all the clauses in $\mathcal{C}$. Since the construction of the total Roman bondage instance is straightforward from a 3-SAT instance, the size of the total Roman bondage
instance is bounded above by a polynomial function of the size of 3-SAT
instance. Consequently, we obtain a polynomial transformation.
\end{proof}

\section{ Results on total Roman bondage number}

In this section, we establish sharp bounds for the total Roman bondage number of a graph under various conditions and determine its exact value in specific cases.

\begin{lemma}
Let $G$ be a graph and $B$ a $b_{tR}(G)$-set. Then
$$\gamma_{tR}(G)+1 \leq \gamma_{tR}(G-B)\leq \gamma_{tR}(G)+2.$$
Both bounds are sharp.
\end{lemma}
\begin{proof}
Since $B$ is a $b_{tR}(G)$-set, $\gamma_{tR}(G - B) > \gamma_{tR}(G)$, and thus $\gamma_{tR}(G - B) \geq \gamma_{tR}(G) + 1$.

For the upper bound, let $e \in B$. As $B$ is a $b_{tR}(G)$-set, $\gamma_{tR}((G - B) + e) = \gamma_{tR}(G)$. By Proposition \ref{Pro}, 
$$\gamma_{tR}(G - B) \leq \gamma_{tR}((G - B) + e) + 2 = \gamma_{tR}(G) + 2.$$

Sharpness of the upper bound holds for bistars $S_{p,q}$ with $q \geq p \geq 2$, while the lower bound is sharp for $S_{1,q}$ with $q \geq 2$.
\end{proof}

\begin{theorem}\label{m1}
Let $G$ be a  graph $\gamma _{tR}(G)=\gamma _{t}(G)+2$. Then $b_{tR}(G)\leq b_{t}(G)$.
\end{theorem}
\begin{proof}
Let $B$ be a $b_{t}(G)$-set. Suppose $\Delta (G-B)=n-1$. Then $\Delta (G)=n-1$ and so $\gamma _t(G)=2, \gamma _{tR}(G)=3$, contradicting $\gamma_{tR}(G) = \gamma_t(G) + 2$. 

If $\Delta(G - B) = 1$, then $G - B$ is a disjoint union of $K_2$'s, yielding $\gamma_{tR}(G - B) = n > \gamma_{tR}(G)$, so $b_{tR}(G) \leq |B| = b_t(G)$.

Now assume $2 \leq \Delta(G - B) \leq n - 2$. By Parts 4, 7, and 8 of Proposition \ref{AA},
$$\gamma_{tR}(G - B) \geq \gamma_t(G - B) + 2 > \gamma_t(G) + 2 = \gamma_{tR}(G).$$
Thus, $b_{tR}(G) \leq |B| = b_t(G)$.
\end{proof}

\begin{corollary}\label{ccc}
Let $G$ be a graph with $\gamma_{tR}(G) = 4$ (equivalently, $\Delta(G) \leq n-2$ and adjacent vertices $u,v$ exist with $N[u] \cup N[v] = V(G)$). Then $b_{tR}(G) = b_t(G)$.
\end{corollary}
\begin{proof}
By Proposition \ref{4R} and Remark \ref{Remarks}(5), $\gamma_{tR}(G) = 4$ and $\gamma_t(G) = 2$. Since $\gamma_{tR}(G) = \gamma_t(G) + 2 = 2\gamma_t(G)$, the result follows from Lemma \ref{m2}(6) and Theorem \ref{m1}.
\end{proof}

By \cite[Theorem 5]{kul}, $b_t(K_{m,n}) = m$ for $2 \leq m \leq n$. Thus, Corollary \ref{ccc} yields:

\begin{corollary}
For $2\leq m\leq n$, $b_{tR}(K_{m,n})=m$.
\end{corollary}

\begin{proposition}
Let $G$ be a graph with $\gamma_{tR}(G)=3 \beta(G)$. Then $b_{tR}(G) \geq \max \{\delta(G), b(G)\}$.
\end{proposition}
\begin{proof}
By Parts 1 and 5 of Proposition \ref{AA}, $\gamma_{tR}(G) \leq 3\gamma(G) \leq 3\beta(G)$, so equality implies $\gamma_{tR}(G) = 3\gamma(G)$. Thus $b_{tR}(G) \geq b(G)$ by Lemma \ref{m2}(8).

If $\delta(G) = 1$, the result holds. Assume $\delta(G) \geq 2$ and let $B \subseteq E(G)$ with $|B| \leq \delta(G) - 1$. Then $\delta(G - B) \geq 1$. Again by Parts 1 and 5 of Proposition \ref{AA},
$$\gamma_{tR}(G) \leq \gamma_{tR}(G - B) \leq 3\gamma(G - B) \leq 3\beta(G - B) \leq 3\beta(G) = \gamma_{tR}(G),$$
so $\gamma_{tR}(G - B) = \gamma_{tR}(G)$. Hence, $b_{tR}(G) \geq \delta(G)$.
\end{proof}

\begin{lemma}\label{bt1}
\begin{enumerate}
\item Let $R$ be one of the following graphs:
 \begin{itemize}
\item[$\bullet$] A star $S_r$ with center $v$ and $r \geq 2$ leaves;
\item[$\bullet$] A spider $S(0,r)$ with head $v$ and $r \geq 2$ feet;
\item[$\bullet$] A cycle $C_n$ ($n \geq 5$) with vertex $v \in V(C_n)$;
\item[$\bullet$] A path $P_n$ ($n \geq 5$) with support vertex $v$.
 \end{itemize}
Let $H$ be a graph with $V(H) \cap V(R) = \emptyset$ and $u \in V(H)$ such that some $\gamma_{tR}(H)$-function $f$ satisfies $f(u) > 0$ ($f(u) = 2$ if $R = S(0,r)$). Let $G$ arise from $H \cup R$ by adding $s$ edges including $uv$. Then $b_{tR}(G) \leq s$.

\item If $G$ has adjacent support vertices $u,v$ with $v$ having $r \geq 2$ leaf neighbors, then $b_{tR}(G) \leq \deg(v) - r$.
\end{enumerate}
\end{lemma}
\begin{proof}
(1) Note that $\gamma_{tR}(S_r) = 3$, $\gamma_{tR}(S(0,r)) = 2r + 1$, $\gamma_{tR}(C_n) = n$, $\gamma_{tR}(P_n) = n$, and $\gamma_{tR}(H \cup R) = \gamma_{tR}(H) + \gamma_{tR}(R)$. Since $f(u) > 0$, define TRDFs on $G$ extending $f$ as follows:

For $R = S_r$:
$$g_1(v) = 2, \quad g_1(x) = 0 \ (x \in L(v)), \quad g_1(x) = f(x) \ (x \in V(H)).$$

For $R = S(0,r)$:
$$g_2(v) = 0, \quad g_2(x) = 1 \ (x \in V(R) \setminus \{v\}), \quad g_2(x) = f(x) \ (x \in V(H)).$$

For $R = C_n$:
$$g_3(v) = 2, \ g_3(x) = 0 \ (x \in N_R(v)), \ g_3(x) = 1 \ (x \in V(R) \setminus N_R[v]), \  g_3(x) = f(x) \ (x \in V(H)).$$

For $R = P_n$ (analogous to $C_n$).

Each $g_i$ is a TRDF on $G$ with weight $< \gamma_{tR}(H \cup R)$, so $\gamma_{tR}(G) < \gamma_{tR}(H \cup R)$ and $b_{tR}(G) \leq s$.

(2) Let $H = G \setminus (L(v) \cup \{v\})$. Apply Remark \ref{Remarks}(1) and Part 1.
\end{proof}

With the aid of the subsequent theorem, the total Roman bondage number can be computed for a broad class of trees. To this end, recall that trees can be rooted at a vertex $r$.  For $x\in V(T)$ the length of the longest upward path from $x$ to a leaf is called the height of $x$.  

\begin{theorem}
Let $T$ be a rooted tree with root $r$, and let $x_1, \dots, x_t$ be vertices with height one. For each $i$, let $v_i$ be the non-leaf neighbour of $x_i$, with $v_i$ having $k_i$ leaves and $x_i$ having $r_i$ leaves.
\begin{enumerate}
\item If $k_i \geq 1$ and $r_i \geq 2$ for some $i$, then $b_{tR}(T) = 1$.
\item If $k_i \geq 2$ and $r_i = 1$ for some $i$, then $b_{tR}(T) \leq \deg(v_i) - k_i$.
\item For integers $t \geq 3$, $2 \leq k \leq t-1$, $b_{tR}(S(k,t)) = t - k$.
\end{enumerate}
\end{theorem}
\begin{proof}
(1) Neighbors of $x_i$ except $v_i$ are leaves, so Lemma \ref{bt1}(2) applies.

(2) Direct from Lemma \ref{bt1}(2).

(3) Root $S(k,t)$ at its head. Part 2 gives $b_{tR}(S(k,t))\leq t - k$. Removing fewer than $t - k$ non-leaf-incident edges from the head preserves $\gamma_{tR}$.
\end{proof}

\begin{theorem}\label{3}
If $\gamma _{tR}(G)=4$, then $b_{tR}(G)\leq n-1$.
\end{theorem}
\begin{proof}
By Proposition \ref{4R}, $\Delta(G) \leq n-2$ and adjacent $u,v$ exist with $N[u] \cup N[v] = V(G)$, so $\gamma(G) = \gamma_t(G) = 2$. Then $b_t(G) \leq b(G)$ by Lemma \ref{m2}(1). By \cite[Theorem 7]{refrefref}, for $H$ with $\gamma(H) \geq 2$, $b(H) \leq (\gamma(H) - 1)\Delta(H) + 1$. With Corollary \ref{ccc},
$$b_{tR}(G) = b_t(G) \leq b(G) \leq \Delta(G) + 1 \leq n - 1.$$
\end{proof}

\begin{lemma}\label{sub}
If $H$ is a spanning subgraph of $G$ obtained by removing $k$ edges with $\gamma_{tR}(G) = \gamma_{tR}(H)$, then
$$b_{tR}(H) \leq b_{tR}(G) \leq b_{tR}(H) + k.$$
\end{lemma}
\begin{proof}
Let $B$ be a $b_{tR}(G)$-set. Then $\gamma_{tR}(H - (B\cap E(H))) \geq \gamma_{tR}(G - B) > \gamma_{tR}(G) = \gamma_{tR}(H)$, so $b_{tR}(H) \leq |B\cap E(H)| \leq |B|$.

Let $F = E(G) \setminus E(H)$ and $B$ a $b_{tR}(H)$-set. Then $\gamma_{tR}(G - (F \cup B)) = \gamma_{tR}(H - B) > \gamma_{tR}(H) = \gamma_{tR}(G)$, so $b_{tR}(G) \leq |F \cup B| = b_{tR}(H) + k$.
\end{proof}

\begin{theorem}\label{tb1}
$b_{tR}(G) = 1$ if and only if $G$ has an edge $uv$ with $G - uv$ isolate-free and every $\gamma_{tR}(G)$-function $f = (V_0, V_1, V_2)$ satisfies:
\begin{enumerate}
\item[(i)] $G[V_1 \cup V_2] - uv$ has an isolated vertex; or
\item[(ii)] $u \in V_2$ and $v \in \mathrm{epn}(u, V_2) \cap V_0$, or vice versa.
\end{enumerate}
\end{theorem}
\begin{proof}
($\Rightarrow$) Let $b_{tR}(G) = 1$, so some $uv$ satisfies $\gamma_{tR}(G - uv) > \gamma_{tR}(G)$ and $G - uv$ isolate-free. Let $f = (V_0, V_1, V_2)$ be a $\gamma_{tR}(G)$-function. Suppose $G[V_1 \cup V_2] - uv$ has no isolated vertices. If $u,v \in V_1 \cup V_2$ or $u,v \in V_0 \cup V_1$, then $f$ remains a TRDF on $G - uv$, a contradiction. Assume without loss of generality $u \in V_2$, $v \in V_0$. If $v \notin \mathrm{epn}(u, V_2)$, then $v$ has neighbor $w \in V_2 \setminus \{u\}$, so $f$ is a TRDF on $G - uv$, a contradiction. Thus (ii) holds if (i) fails.

($\Leftarrow$) Assume such $uv$ exists and every $\gamma_{tR}(G)$-function satisfies (i) or (ii), with $G - uv$ isolate-free. Suppose $\gamma_{tR}(G - uv) = \gamma_{tR}(G)$. Let $g = (V^g_0, V^g_1, V^g_2)$ be a $\gamma_{tR}(G - uv)$-function; then $g$ is also $\gamma_{tR}(G)$-function, so satisfies (i) or (ii). But $G[V^g_1 \cup V^g_2] - uv$ has no isolated vertices (as $g$ dominates $G - uv$ totally). Assume without loss of generality $u \in V^g_2$, $v \in \mathrm{epn}(u, V^g_2) \cap V^g_0$. Then $v \in V^g_0$ has neighbor $w \in V^g_2 \cap N_{G - uv}(v)$, so $v \notin \mathrm{epn}(u, V^g_2)$, a contradiction. Thus $\gamma_{tR}(G - uv) > \gamma_{tR}(G)$.
\end{proof}

\begin{corollary}
If $\delta(G) \geq 2$ and $G$ has a unique $\gamma_{tR}(G)$-function, then $b_{tR}(G) = 1$.
\end{corollary}
\begin{proof}
Let $f=(V_0,V_1,V_2)$ be the unique $\gamma_{tR}(G)$-function. Since $G \in \mathcal{B}$, we have $V_2 \neq \emptyset$.
Assume that $v \in V_2$.
If $\mathrm{epn}(v,V_2) \cap V_0  = \emptyset$, then
the function  $g=(V_0,V_1 \cup \{v\},V_2-\{v\})$
defines a TRDF on $G$ of weight $\gamma_{tR}(G)-1$ which is a contradiction.
Hence $\mathrm{epn}(v,V_2) \cap V_0 \neq \emptyset$.
Suppose that $u \in \mathrm{epn}(v,V_2) \cap V_0 $. Since $\delta (G)\geq 2$, $G-uv$ has no isolated vertices and so the result follows from Theorem~\ref{tb1}, because $f$ is unique.
\end{proof}

\begin{proposition}
If $\{u_1, \dots, u_k\}$ ($k \geq 4$) forms a clique in $G$ and $G \setminus \{u_1, \dots, u_k\}$ is isolate-free, then
$$b_{tR}(G) \leq \sum_{i=1}^k \deg(u_i) - \frac{k(k+1)}{2}.$$
\end{proposition}
\begin{proof}
Let $H$ arise by removing edges incident to the cycle $(u_1, \dots, u_k)$ except cycle edges. Suppose $\gamma_{tR}(H) = \gamma_{tR}(G)$ with TRDF $f$ on $H$ (hence on $G$). By Proposition~\ref{ahah1}(1), $\sum_{i=1}^k f(u_i) = k$. Define $g(u_1) = 2$, $g(u_2) = 1$, $g(u_i) = 0$ ($3 \leq i \leq k$), $g(x) = f(x)$ otherwise. Then $g$ is a TRDF on $G$ with weight $< \gamma_{tR}(G)$, a contradiction. Thus $\gamma_{tR}(H) > \gamma_{tR}(G)$, and
$$b_{tR}(G) \leq \sum_{i=1}^k (\deg(u_i) - (k-1))+\frac{k(k-1)}{2}-k=\sum_{i=1}^k \deg(u_i) - \frac{k(k+1)}{2}.$$
\end{proof}

\begin{lemma}\label{Prop20}
Let $G$ ($n \geq 5$) have strong support vertex $v$ with non-leaf neighbor $a$ such that $N_G(a) \subseteq V(G) \setminus (N_G(v) \cup N_G(b))$ for all $b \in N_G(a)$. Then $b_{tR}(G) \leq n - 4$.
\end{lemma}
\begin{proof}
Let $B_1$ be edges incident to $v$ except leaf-incident, $H_1 = G - B_1$, $R$ the component of $H_1$ containing $v$. If some $\gamma_{tR}(H_1)$-function assigns positive value to some $u \in N_G(v) \setminus L(v)$, then Lemma \ref{bt1}(1) gives $b_{tR}(G) \leq |B_1|$.

Otherwise, all such functions assign $0$ to all $u \in N_G(v) \setminus L(v)$. For any $\gamma_{tR}(H_1)$-function $f_1'$, $f_1'(a) = 0$, $a$ has no leaf, so some $b \in N_{H_1}(a)$ has $f_1'(b) = 2$.

Let $B_2$ be edges incident to $a,b$ except pendant edges of $b$ and $bu$ ($u \in N_G(v)$), $H_2 = G - (B_1 \cup B_2)$. If some $\gamma_{tR}(H_2)$-function assigns positive value to $u \in N_G(v) \setminus L(v)$, then $b_{tR}(G) \leq |B_1 \cup B_2|$.

Otherwise, for any $\gamma_{tR}(H_2)$-function $f_2'$, $f_2'(u) = 0$ for all $u \in N_G(v) \setminus L(v)$, $N_{H_2}(a) = \{b\}$, $f_2'(a) = 0$ imply $f_2'(b) = 2$. For totality of vertices $b$ and $v$, leaves $c \in L(b)$, $d \in L(v)$ exist with $f_2'(c) = f_2'(d) = 1$. Define $g(v) = 2$, $g(x) = 0$ ($x \in L(v) \cup L(b)$), $g(a) = 1$, $g(x) = f_2'(x)$ otherwise: $g$ is a TRDF on $G$ with weight $< \gamma_{tR}(H_2)$, so $b_{tR}(G) \leq |B_1 \cup B_2|$. 

Thus $b_{tR}(G) \leq |B_1 \cup B_2| \leq (\deg(v) - |L(v)|) + (\deg(a) - 2) + (\deg(b) - |N_G(v) \cap N_G(b)| - |L(b)|)$. By assumptions ($|L(v)| \geq 2$, $|L(b)| \geq 0$, $N_G(a) \subseteq V(G) \setminus (N_G(v) \cup N_G(b))$),

\begin{align*}
b_{tR}(G) &\leq \deg(v)-2+ (n-|N_G(v)\cup N_G(b)|)-2+\deg(b)-|N_G(v) \cap N_G(b)| \\
&\leq \deg(v)-4+ n-\deg(v)-\deg(b)+|N_G(v) \cap N_G(b)|+\deg(b)-|N_G(v) \cap N_G(b)|\\
& =n-4.
\end{align*}
\end{proof}

The following is a straightforward consequence of Lemma \ref{Prop20}.
\begin{corollary}
If $G$ ($n \geq 5$) has a strong support vertex $v$ and one of:
\begin{enumerate}
\item A non-leaf neighbor for $v$ in no triangle;
\item $\mathrm{girth}(G) \geq 4$;
\item $G$ a tree,
\end{enumerate}
then $b_{tR}(G) \leq n - 4$.
\end{corollary}

\begin{proposition}\label{Pro21}
Let $G$ (connected with $n$ vertices) have cycle $C = (v_1, \dots, v_k)$ ($k \geq 5$) with $r$ chords, $G - V(C)$ isolate-free, and $N_G(v_i) \cap N_G(v_j) \subseteq V(C)$ for $i \neq j$. Then $$b_{tR}(G) \leq r + n - k - 1.$$
\end{proposition}

\begin{proof}
Let $A = (\bigcup_{i=1}^k N_G(v_i)) \setminus V(C)$, $B_1$ edges incident to $C$ except $C$-edges. By Proposition \ref{ahah1}, $\gamma_{tR}(G - B_1) = k + \gamma_{tR}(G - V(C))$. Let $f$ be a $\gamma_{tR}(G - B_1)$-function. If $f(u) > 0$ for some $u \in A$, then Lemma \ref{bt1}(1) gives $b_{tR}(G) \leq |B_1|$.

Otherwise, no such assignment. For $a \in A$, $f(a) = 0$ implies there exists $b \in N_{G - B_1}(a)$ with $f(b) = 2$. Let $B_2 = \{bx \mid x \notin A\}$. Then $B_1 \cap B_2 = \emptyset$, since $f(b)=2$.Suppose $\gamma_{tR}(G - (B_1 \cup B_2)) = \gamma_{tR}(G)$. Then since $\gamma_{tR}(G-(B_1\cup B_2))\geq \gamma_{tR}(G-B_1)\geq \gamma_{tR}(G)$,  any $\gamma_{tR}(G-(B_1\cup B_2))$-function $g=(V_0, V_1, V_2)$ is also a $\gamma_{tR}(G-B_1)$-function. Thus $g(u)=0$ for all $u\in A$ hence for neighbors of $b$ in $G-(B_1\cup B_2)$, making $b$ isolate in $V_1 \cup V_2$, a contradiction. Thus $\gamma_{tR}(G - (B_1 \cup B_2)) > \gamma_{tR}(G)$.

Since $N_G(v_i) \cap N_G(v_j) \subseteq V(C)$, $|B_1| = r + |A|$. Also $A \cup N_G(b) \subseteq V(G) \setminus (V(C) \cup \{b\})$, so
$$b_{tR}(G) \leq |B_1| + |B_2| \leq r + |A| + (\deg(b) - |N_G(b) \cap A|) \leq r + |A \cup N_G(b)| \leq r + n - k - 1.$$
\end{proof}

The following corollary follows from Proposition \ref{Pro21} and definition of girth of a graph.
\begin{corollary}
If $G$ is a connected graph with $\mathrm{girth}(G) \geq 5$ which has girth-cycle $C$ with $G - V(C)$ isolate-free, then 
$$b_{tR}(G) \leq n - \mathrm{girth}(G) - 1.$$
\end{corollary}

\begin{proposition}
Let $G$ be a connected graph which admits minimum $k$ edges whose removal yields components $G_1, G_2$ with $\delta(G_1) \geq 2$, $\delta(G_2) \geq 1$. Then 
$$b_{tR}(G) \leq 3 \Delta(G)+k-4.$$
\end{proposition}
\begin{proof}
Let $F = \{e_1, \dots, e_k\}$ separate as described, $e_1 = uv$ ($u$ in $G_1$, $v$ in $G_2$). Since $\delta(G_1) \geq 2, \delta(G_2)\geq 1$, $\Delta (G)\geq 2$ and
there exist $u_1,u_2 \in N_{G_1}(u)$ and $v_1 \in N_{G_2}(v)$. Define $E_1$ as edges incident to $u,u_1,u_2$ in $G_1$ except $uu_1, uu_2, u_1u_2, ux$ ($N_{G_1}(x) = \{u,u_1,u_2\}$); $E_2$ incident to $v,v_1$ in $G_2$ except $vv_1, vy$ ($N(y) = \{v,v_1\}$), pendant of $v,v_1$.  Furthermore assume
$$A_1=\{x\in V(G) \mid N_{G_1}(x)=\{u,u_1, u_2\}\}, A_2=\{y \in V(G) \mid  N_{G_2}(y)=\{v,v_1\}\},$$
and
$$\gamma_1=\gamma_{tR}(G_2[\{v,v_1\}\cup L(v)\cup L(v_1)\cup A_2]), \gamma_2=\gamma_{tR}(G[\{u,u_1, u_2,v,v_1\}\cup L(v)\cup L(v_1)\cup A_1 \cup A_2]).$$

If $b_{tR}(G)\leq k$, then the proof has been done. Assume $b_{tR}(G)>k$, so
$\gamma_{tR}(G)=\gamma_{tR}(G_1)+\gamma_{tR}(G_2)$. Suppose to the contrary that $b_{tR}(G) > 3 \Delta(G)+k-4$. Then
$$|F \cup E_1| \leq \deg_{G_1} (u)+ \deg_{G_1} (u_1)+ \deg_{G_1} (u_2)+k-4 \leq 3 \Delta(G)+k-4< b_{tR}(G)$$
and
$$|F \cup E_2| \leq \deg_{G_2} (v)+ \deg_{G_2} (v_1)+k-2 \leq 2 \Delta(G)+k-2= (3\Delta (G) +k-4)+(2-\Delta(G)) < b_{tR}(G),$$
yielding
\begin{align*}
 \gamma_{tR}(G)&=\gamma_{tR}(G_1-(\{u,u_1,u_2\}\cup A_1))+3+\gamma_{tR}(G_2),\\
 \gamma_{tR}(G)&=\gamma_{tR}(G_1)+\gamma_{tR}(G_2-E_2).
\end{align*}
Adding gives
$$2\gamma_{tR}(G)=\gamma_{tR}(G_1-(\{u,u_1,u_2\}\cup A_1))+\gamma_{tR}(G_2-E_2)+3+\gamma_{tR}(G).$$
Therefore
$$\gamma_{tR}(G)=\gamma_{tR}(G_1-(\{u,u_1,u_2\}\cup A_1))+\gamma_{tR}(G_2-(\{v,v_1\}\cup L(v)\cup L(v_1)\cup A_2))+\gamma_1+3,$$
while obviously
$$\gamma_{tR}(G)\leq \gamma_{tR}(G_1-(\{u,u_1,u_2\}\cup A_1))+\gamma_{tR}(G_2-(\{v,v_1\}\cup L(v)\cup L(v_1)\cup A_2))+\gamma_2.$$
Thus $\gamma_1+3 \leq \gamma_2$. Cases contradict:
\begin{itemize}
  \item[$\bullet$] $A_2=L(v)=L(v_1)=\emptyset$: $\gamma_1=2 ,\gamma_2=4$.
  \item[$\bullet$] $A_2=L(v_1)=\emptyset, L(v)\neq \emptyset$ or $A_2\neq \emptyset, L(v_1)=\emptyset$: $\gamma_1=3 ,\gamma_2=4$.
  \item[$\bullet$] $L(v_1)\neq \emptyset, L(v)=\emptyset$: $\gamma_1=3 ,\gamma_2=5$.
  \item[$\bullet$] $L(v)\neq \emptyset, L(v_1)\neq \emptyset$: $\gamma_1=4 ,\gamma_2=6$.
\end{itemize}
\end{proof}

{\bf Acknowledgement.}
The authors are deeply grateful to the
referees for carefully reading of the manuscript and helpful
suggestions.




\end{document}